\documentclass{amsart}
\usepackage{cases}
\usepackage{amsfonts}
\usepackage{mathrsfs}
\usepackage{amsmath}
\usepackage{amsfonts,amssymb,amsmath,amsthm}
\usepackage{url}
\usepackage{enumerate}
\usepackage{amscd}
\usepackage{mathrsfs,graphicx}
\usepackage{latexsym}
\usepackage[all]{xy}
\usepackage{epsfig}

\newtheorem*{thm*}{Theorem}
\newtheorem{thm}{Theorem}[section]
\newtheorem{lem}[thm]{Lemma}
\newtheorem{prop}[thm]{Proposition}
\newtheorem{cor}[thm]{Corollary}
\theoremstyle{definition}
\newtheorem{df}[thm]{Definition}
\newtheorem{exa}[thm]{Example}
\newtheorem{rem}[thm]{Remark}

\newtheorem*{ques*}{Question}
\numberwithin{equation}{section}
\newcommand{\Ext}{\mbox{\rm Ext}}
\newcommand{\Hom}{\mbox{\rm Hom}}

\newcommand{\Coker}{\mbox{\rm Coker}}
\newcommand{\Ker}{\mbox{\rm Ker}}

\author{Gang Yang}\address{Department of Mathematics, Lanzhou Jiaotong University, Lanzhou {\rm 730070}, P. R. China} \email{yanggang10@gmail.com}
\author{Rui-Juan Du}\address{Department of Computer Science, Gansu Political Science and Law Institute, Lanzhou {\rm 730070},
P. R. China} \email{duruijuan1126@126.com}
\thanks{This work was partly supported
by NSF of China (Grant No. 11101197, 11201376, 11301240) and the
Program of Science and Technique of Gansu Province (No.
145RJZA079).}

\begin{document}

\title[On Cotorsion pairs of chain complexes]{\large On Cotorsion pairs of chain complexes}

\begin{abstract}
In the paper we first construct a new cotorsion pair, in the
category of chain complexes, from two given cotorsion pairs in the
category of modules, and then we consider completeness of  such
pairs under certain conditions. \\
\end{abstract}
\maketitle

\section {\large\bf Introduction}
Cotorsion pairs (or cotorsion theories) were invented by Salce in
his study of abelian groups in \cite{Sa}. However, the concept
readily generalized to any abelian category, and its importance in
homological algebra has been shown by its use in the proof of the
flat cover conjecture \cite{BBE}. The flat cover conjecture was
positively settled by showing that the famous cotorsion pair
$(\mathcal{F}, \mathcal{C})$ is complete, where $\mathcal{F}$
denotes the class of flat modules and $\mathcal{C}$ denotes the
class of cotorsion modules. On the other hand, there is a lot of
interest in the complete cotorsion pairs in the category of chain
complexes. It not only is used to show  the existence of certain
covers and envelopes in the category of chain complexes \cite{AERO},
but also is closely related to Quillen model  structures and also to
the existence of certain adjoints. In fact, a famous result of Hovey
\cite{Hov} says that a Quillen model  structure on any abelian
category $\mathcal{C}$ is equivalent to two complete cotorsion pairs
in $\mathcal{C}$ which are compatible in a precise way. One of the
upshots of this result  was that the study of cotorsion pairs in the
category of chain complexes attracted more attentions. Besides this,
a recent result of a group of authors \cite[Theorem 3.5]{BEIJR}
shows that there is a tight connection between the complete
cotorsion pairs in the category of chain complexes of modules and
the existence of adjoint functors on the corresponding homotopy
categories. Hence, there has been several attempts to get (complete)
cotorsion pairs in Ch($R$), the category of chain complexes over a
ring $R$, from ones in $R$-Mod, see e.g. \cite{AERO}, \cite{AH},
\cite{BEIJR}, \cite{EER}, \cite{GR}, \cite{G04}, \cite{G08},
\cite{YL}.

Our main goal in this paper is first to construct a new cotorsion
pair in Ch($R$) from  two given cotorsion pairs in $R$-Mod, and then
to consider completeness of our constructed cotorsion pairs. More
specifically, given two classes of $R$-modules $\mathcal{U}$ and
$\mathcal{X}$, where $\mathcal{U}\subseteq\mathcal{X}$, we have the
following classes of chain complexes in Ch($R$).
\begin{itemize}
  \item $dw\widetilde{\mathcal{U}}$ is the class of all chain complexes  $U$
with each degree $U_n\in \mathcal{U}$.
  \item $ex\widetilde{\mathcal{U}}$ is the class of all exact  chain complexes  $U$
with each degree $U_n\in \mathcal{U}$.
  \item $\widetilde{\mathcal{U}}$ is the class of all exact  chain complexes  $U$ with
each cycle $Z_nU\in \mathcal{U}$.
  \item $\widetilde{\mathcal{U}}_\mathcal{X}$ is the class of all exact  chain complexes  $U$ with
 each degree $U_n\in \mathcal{U}$ and each cycle $Z_nU\in \mathcal{X}$.
\end{itemize}
Then theorems \ref{thm3.3} and \ref{thm3.4} say that if
$(\mathcal{U}, \mathcal{V})$ and $(\mathcal{X}, \mathcal{Y})$ are
two cotorsion pairs  with $\mathcal{U}\subseteq\mathcal{X}$ in
$R$-Mod, then $(\widetilde{\mathcal{U}}_\mathcal{X},
(\widetilde{\mathcal{U}}_\mathcal{X})^\bot)$ and
$({^\bot(\widetilde{\mathcal{Y}}_\mathcal{V})},
\widetilde{\mathcal{Y}}_\mathcal{V})$ are cotorsion pairs in
Ch($R$). This result immediately yields a list of cotorsion pairs in
Ch($R$) below, and so our argument gives a unified proof for most of
the existing cotorsion pairs in Ch($R$).
$$(\widetilde{\mathcal{U}},\widetilde{\mathcal{U}}^\bot)
\hspace{1cm}({^\bot\widetilde{\mathcal{V}}},\widetilde{\mathcal{V}})
\hspace{1cm}(ex\widetilde{\mathcal{U}},(ex\widetilde{\mathcal{U}})^\bot)
\hspace{1cm}({^\bot(ex\widetilde{\mathcal{V}}}),
ex\widetilde{\mathcal{V}})$$ Theorems \ref{rhm3.9} and \ref{rhm3.11}
say the following: Assume that $(\mathcal{U}, \mathcal{V})$ is a
hereditary cotorsion pair in $R$-Mod. Then  the cotorsion pair
$(dw\widetilde{\mathcal{U}}, (dw\widetilde{\mathcal{U}})^\bot)$ is
complete if and only if the cotorsion pair
$(ex\widetilde{\mathcal{U}}, (ex\widetilde{\mathcal{U}})^\bot)$ is
complete; and  the cotorsion pair
$({^\bot(dw\widetilde{\mathcal{V}})}, dw\widetilde{\mathcal{V}})$ is
complete if and only if the cotorsion pair
$({^\bot(ex\widetilde{\mathcal{V}})}, ex\widetilde{\mathcal{V}})$ is
complete. In the end of this paper, we consider cogenerated sets of
such pairs under certain conditions.

\section {\large\bf Preliminaries}
Throughout this paper, let $R$ be an associative ring with 1,
$R$-Mod  the category of left $R$-modules and $\text{Ch}(R)$ the
category of chain complexes of left $R$-modules. We denote a chain
complex $\cdots \rightarrow C_{n+1}\xrightarrow{\delta^C_{n+1}}
C_{n}\xrightarrow{\delta^C_n}C_{n-1}\rightarrow\cdots$ by
$(C,\delta)$ or simply $C$. The $n$th cycle of a chain complex $C$
is defined as $\Ker(\delta^C_n)$ and is denoted by $Z_nC$, the $n$th
boundary is $\text{Im}(\delta^C_{n+1})$ and is denoted by $B_nC$,
the $n$th homology is the module $H_nC=Z_nC/B_nC.$ A complex $C$ is
said to be exact if $H_nC=0$ for all $n\in\mathbb{Z}$.

We let $S^n(M)$ denote the chain complex with all entries 0 except
$M$  in degree $n$, and let $D^n(M)$ denote the chain complex $C$
with $C_n=C_{n-1}=M$ and all other entries 0, all differentials 0
except $\delta_n=1_M$. The suspension of a chain complex $C$,
denoted $\Sigma C$, is the chain complex given by $(\Sigma
C)_n=C_{n-1}$ and $\delta^{\Sigma C}_n=-\delta^C_{n-1}$. The chain
complex $\Sigma(\Sigma C)$ is denoted $\Sigma^2C$ and inductively we
define $\Sigma^nC$ for all $\in\mathbb{Z}$.

Given two chain complexes $X$ and $Y$ we define $\text{Hom}(X, Y)$
to be the complex of $\mathbb{Z}$-modules with $n$th degree
$\text{Hom}(X, Y)_n=\Pi_{i\in\mathbb{Z}}\text{Hom}(X_i, Y_{i+n})$
and differential $\delta_n$ satisfying
$(\delta_n(f))_i=\delta^Y_{i+n}f_i-(-1)^nf_{i-1}\delta^X_i$. This
gives a functor $\text{Hom}(X, -):
\text{Ch}(R)\rightarrow\text{Ch}(\mathbb{Z})$ which is left exact,
and exact if $X_n$ is projective for all $n$. Similarly, the
contravariant  functor $\text{Hom}(-, Y)$ sends right exact equences
to left exact sequences and is exact if $Y_n$ is injective for all
$n$. Note that the category Ch($R$) is a Grothendieck category with
a projective generator, and so it has enough projectives. Recall
that a Grothendieck category is an abelian category with a generator
and  with the property that the  direct limits are exact.

Recall that $\Ext^1_{\text{Ch}(R)}(X, Y)$ is the group of
(equivalence classes) of short exact sequences $0\rightarrow
Y\rightarrow Z\rightarrow X\rightarrow0$ under the Baer sum. We let
$\Ext^1_{dw}(X, Y)$ be the subgroup of $\Ext^1_{\text{Ch}(R)}(X, Y)$
consisting of those short exact sequences which are split in each
degree. We often make use of the following standard fact.

\begin{lem}\label{lem2.1} For two chain complexes $X$ and $Y$, we have
$$\Ext^1_{dw}(X, \Sigma^{-n-1}Y)\cong H_n\Hom(X, Y)=\text{Ch}(R)(X,
\Sigma^{-n}Y)/\sim,$$ where $\sim$ is chain homotopy.
\end{lem}

In particular, for two chain complexes $X$ and $Y$, $\Hom(X, Y)$ is
exact if and only if for any $n\in\mathbb{Z}$, any
$f:\Sigma^{n}X\rightarrow Y$ is homotopic to 0 (or if and only if
any $f:X\rightarrow \Sigma^{-n}Y$ is homotopic to 0).

\begin{df} A pair $(\mathcal{A},
\mathcal{B})$ in an abelian category $\mathcal{C}$ is called a
cotorsion pair if the following conditions hold:
\begin{enumerate}
                \item $\Ext^1_\mathcal{C}(A, B)=0$ for all $A\in\mathcal{A}$ and $B\in\mathcal{B}$;
                \item If $\Ext^1_\mathcal{C}(A, X)=0$ for all $A\in\mathcal{A}$ then $X\in\mathcal{B}$;
                \item If $\Ext^1_\mathcal{C}(X, B)=0$ for all $B\in\mathcal{B}$ then $X\in\mathcal{A}$.
\end{enumerate}
\end{df}

We think of a cotorsion pair $(\mathcal{A}, \mathcal{B})$ as being
$\lq\lq$orthogonal with respect to $\Ext^1_\mathcal{C}$". This is
often expressed with the notation $\mathcal{A}={^\perp\mathcal{B}}$
and $\mathcal{B}=\mathcal{A}^\perp$. The notion of a cotorsion pair
was first introduced by Salce in \cite{Sa} and rediscovered by
Enochs and coauthors in 1990's. For a good reference on cotorsion
pairs one can refer to \cite{EJ}.

\begin{df} A cotorsion pair $(\mathcal{A}, \mathcal{B})$ in an abelian category $\mathcal{C}$ is
said to have enough projectives if for any object $X\in\mathcal{C}$
there is a short exact sequence $0\rightarrow B\rightarrow
A\rightarrow X\rightarrow 0$ with $A\in \mathcal{A}$ and $B\in
\mathcal{B}$. We say it has enough injectives if it satisfies the
dual statement. If both of these hold we say the cotorsion pair is
complete.
\end{df}

Note that if the category $\mathcal{C}$ has enough injectives and
projectives then a cotorsion pair $(\mathcal{A}, \mathcal{B})$ is
complete if and only if $(\mathcal{A}, \mathcal{B})$ has enough
injectives if and only if $(\mathcal{A}, \mathcal{B})$ has enough
projectives \cite{EJ}.

\begin{df} A cotorsion pair $(\mathcal{A}, \mathcal{B})$ in an abelian category $\mathcal{C}$ is
said to be hereditary, if $\Ext^i_\mathcal{C}(A, B)=0$ for any
object $A\in\mathcal{A}$ and $B\in\mathcal{B}$ and $i\geq1$.
\end{df}

In $R$-Mod, the class of projectives is the left half of an obvious
hereditary complete cotorsion pair while the class of injectives is
the right half of an obvious hereditary complete cotorsion pair.
There are many  nontrivial examples of hereditay complete cotorsion
pairs, which can be found in \cite{GT}. We also need the next two
definitions (see \cite{St}).

\begin{df}\label{df} Let $\mathcal{S}$ be a class of objects of
a Grothendieck category $\mathcal{G}$. An object $X\in\mathcal{G}$
is called $\mathcal{S}$-filtered if there exists a well-ordered
direct system $(X_{\alpha}, i_{\alpha\beta}|\alpha<\beta\leq\sigma)$
indexed by an ordinal number $\sigma$ such that
\begin{enumerate}
  \item $X_0=0$ and $X_{\sigma}=X$,
  \item for each limit ordinal $\mu\leq\sigma$, the direct limit of the subsystem
  $(X_{\alpha}, i_{\alpha\beta}|\alpha<\beta\leq\mu)$  is precisely $X_{\mu}$,
  the direct limit morphisms being $i_{\alpha\mu}: X_{\alpha}\rightarrow X_{\mu}$,
  \item $i_{\alpha\beta}: X_{\alpha}\rightarrow X_{\beta}$ is a monomorphism in $\mathcal{G}$ for each $\alpha<\beta\leq\sigma$,
  \item $\Coker(i_{\alpha,\alpha+1})\in\mathcal{S}$ for each
  $\alpha<\sigma$.
\end{enumerate}
The direct system $(X_{\alpha}, i_{\alpha\beta})$ is then called an
$\mathcal{S}$-filtration of $X$. The class of all
$\mathcal{S}$-filtered objects in $\mathcal{G}$ is denoted by
Filt-$\mathcal{S}$.
\end{df}

\begin{df} A class $\mathcal{F}$ of objects in $\mathcal{G}$ is called deconstructible if there is a set
$\mathcal{S}$ such that $\mathcal{F}=\text{Filt-}\mathcal{S}$.
\end{df}

If $P$ is a projective $R$-module and $x\in P$, then Kaplansky
\cite{K} showed that there exists a countably generated summand of
$P$ which contains $x$. Enochs and L$\acute{o}$pez-Ramos generalized
this ideal and introduced the notion of a Kaplansky class, see
\cite[Definition 2.1]{EL}.

\begin{df} A class $\mathcal{K}$ of $R$-modules is called a
$\kappa$-Kaplansky class if there exists a cardinal number $\kappa$
such that for every $M\in\mathcal{K}$ and for any subset $S\subseteq
M$ with $\text{Card}(S)\leq\kappa$, there exists a submodule $N$ of
$M$ that contains $S$ with the property that
$\text{Card}(N)\leq\kappa$ and both $N$ and $M/N$ are in
$\mathcal{K}$. We say that $\mathcal{K}$ is a Kaplansky class if it
is a $\kappa$-Kaplansky class for some regular cardinal $\kappa$.
\end{df}

Let $C$ be a chain complex in Ch($R$). By the cardinality of $C$,
$\text{Card}(C)$, we mean
$\text{Card}(\coprod_{n\in\mathbb{Z}}C_n)$. By a subset $S$ of $X$
we mean a family $(S_n)_{n\in\mathbb{Z}}$ such that $S_n$ is a
subset of $C_n$, for $n\in\mathbb{Z}$. Similarly, we have the notion
of a Kaplansky class of chain complexes. We assume in the paper that
all cardinals are regular, that is, are infinite cardinals which are
not the sum of a smaller number of smaller cardinals.  We let
$\omega$ denote the first limit ordinal.
%%%%%%%%%%%%%%%%%%%%%%%%%%%%%%%%%%%%%%%%%%%%%%%%%%%%%%%%%%%%%%%%%%%%%%%%%%%%%%%%%%%%%%%%%%%%%%%%%%%%%%%%%%%%%%%%%%%%%%%%%%%%

\section*{Acknowledgements}
The authors  thank the referee for his/her careful reading and many
considerable suggestions, which have improved the present paper.

%%%%%%%%%%%%%%%%%%%%%%%%%%%%%%%%%%%%%%%%%%%%%%%%%%%%%%%%%%%%%%%%%%%%%%%%%%%%%%%%%%%%%%%%%%%%%%%%%%%%%%%%%%%%%%%%%%%%%%%%%%%%

\section{\bf Cotortion pairs in the category of chain complexes}\label{ns}

\begin{lem}\label{lem3.1} Let $X$ be a chain complex and $J$ be an
injective cogenerator for $R$-Mod. If every chain map $\alpha:
X\rightarrow S^n(J)$ lifts over $D^n(J)$ for any $n\in\mathbb{Z}$, then $X$ is exact.
\end{lem}
\begin{proof}Let $n$ be an arbitrary integer, we need only to show
exactness of $X$ in degree $n$. Suppose that $t: X_n/B_nX\rightarrow
J$ is a monomorphism. Then it is easy to check that $\alpha:
X\rightarrow S^n(J)$ is a chain map, where $\alpha_i=0$ for
$i\not=n$, and $\alpha_n$ is the composition of homomorphisms
$\pi_n: X_n\rightarrow X_n/B_nX$ and $t$. By hypothesis, there
exists a chain map $\beta: X\rightarrow D^n(J)$ such that the
following diagram is commutative:
$$\xymatrix{ & X\ar@{.>}[dl]|-{\beta}\ar[d]^{\alpha} &   \\
  D^n(J)\ar[r]^{p} & S^n(J) \ar[r]^{ } &0
  }$$
Put $\delta^X_n=\rho_n\pi_n$, where $\rho_n: X_n/B_nX\rightarrow
X_{n-1}$. Thus
$t\pi_n=\alpha_n=p_n\beta_n=\beta_n=\beta_{n-1}\delta^X_n=\beta_{n-1}\rho_n\pi_n$,
and so $t=\beta_{n-1}\rho_n$ since $\pi_n$ is epic. This implies
that $\rho_n$ is a monomorphism. Therefore
$Z_nX=\text{Ker}(\rho_n\pi_n)=\text{Ker}(\pi_n)=B_nX$. This proves
exactness of $X$.
\end{proof}

\begin{df}
Given two classes  of $R$-modules $\mathcal{U}$ and $\mathcal{X}$ in
$R$-Mod with $\mathcal{U}\subseteq\mathcal{X}$. We denote by
$\widetilde{\mathcal{U}}_\mathcal{X}$ the class of all exact chain
complexes $U$ with each degree $U_n\in \mathcal{U}$ and each cycle
$Z_nU\in\mathcal{X}$ in $\text{Ch}(R)$.
\end{df}

 Clearly if we let
$\mathcal{U}$ and $\mathcal{X}$ be certain classes  of modules, we
will get some familiar and interesting classes in $\text{Ch}(R)$.
For example, if  $\mathcal{U}=\mathcal{P}$ is the class of all
projective modules and $\mathcal{X}=\mathcal{G}$ is the class of all
Gorenstein projective modules in $R$-Mod, then
$\widetilde{\mathcal{P}}_\mathcal{G}$ is the class of all complete
projective resolutions of Gorenstein projective modules. (See
\cite{EJ} and \cite{Hol} for Gorenstein ptojective modules).

\begin{thm}\label{thm3.3} Let $(\mathcal{U}, \mathcal{V})$ and $(\mathcal{X},
\mathcal{Y})$ be two cotorsion pairs  with
$\mathcal{U}\subseteq\mathcal{X}$ in $R$-Mod. Then
$(\widetilde{\mathcal{U}}_\mathcal{X},
(\widetilde{\mathcal{U}}_\mathcal{X})^\bot)$ is a cotorsion pair in
$\text{Ch}(R)$ and $(\widetilde{\mathcal{U}}_\mathcal{X})^\bot$ is
the class of all chain complexes $V$ for which each
$V_n\in\mathcal{V}$ and for which each map $U\rightarrow V$ is null
homotopic whenever $U\in\widetilde{\mathcal{U}}_\mathcal{X}$.
\end{thm}

\begin{proof} Let $\widehat{\mathcal{W}}$ denote the  class of all
chain complexes $V$ for which each $V_n\in\mathcal{V}$ and for which
each map $U\rightarrow V$ is null homotopic whenever
$U\in\widetilde{\mathcal{U}}_\mathcal{X}$. It is clear that
$\widehat{\mathcal{W}}$ is closed under taking suspensions. Given
any chain complex $U\in\widetilde{\mathcal{U}}_\mathcal{X}$, and any
$R$-module $V\in\mathcal{V}$, then by \cite[Lemma 3.1]{G04} we have
$\Ext^1_{\text{Ch}(R)}(U, D^{n+1}(V))\cong\Ext^1_R(U_n, V)=0$, which
implies that all disks $D^n(V)$ are contained in
$\widehat{\mathcal{W}}$ whenever $V\in\mathcal{V}$. Similarly, for
any $U\in\widetilde{\mathcal{U}}_\mathcal{X}$, since $U_n/B_nU\cong
Z_{n-1}U\in\mathcal{X}$, we get by \cite[Lemma 4.2]{G08} that
$\Ext^1_{\text{Ch}(R)}(U, S^n(Y))\cong\Ext^1_R(U_n/B_nU, Y)=0$ for
any $R$-module $Y\in\mathcal{Y}$, and so each sphere
$S^n(Y)\in\widehat{\mathcal{W}}$ whenever $Y\in\mathcal{Y}$.

In the following we will show that
$(\widetilde{\mathcal{U}}_\mathcal{X}, \widehat{\mathcal{W}})$ is a
cotorsion pair.

First suppose that $U\in\widetilde{\mathcal{U}}_\mathcal{X}$, and
$W\in\widehat{\mathcal{W}}$. Then any element $0\rightarrow
W\rightarrow T\rightarrow U\rightarrow0$ of
$\Ext^1_{\text{Ch}(R)}(U, W)$ is degreewise split and so is an
element of $\Ext^1_{dw}(U, W)$. But it follows easily from Lemma
\ref{lem2.1} that $\Ext^1_{dw}(U, W)=0$. Thus
$\Ext^1_{\text{Ch}(R)}(U, W)=0$.

Next assume that $\Ext^1_{\text{Ch}(R)}(U, C)=0$ for all
$U\in\widetilde{\mathcal{U}}_\mathcal{X}$, we will show
$C\in\widehat{\mathcal{W}}$. By \cite[Lemma 3.1]{G04}, we have
$\Ext^1_R(A, C_n)\cong\Ext^1_{\text{Ch}(R)}(D^n(A), C)=0$ since
$D^n(A)$ is clearly in $\widetilde{\mathcal{U}}_\mathcal{X}$
whenever $A$ is an $R$-module in $\mathcal{U}$. Thus
$C_n\in\mathcal{V}$. Now let $U\rightarrow C$ be a chain map, where
$U\in\widetilde{\mathcal{U}}_\mathcal{X}$. We would like to show
that it is null homotopic. Clearly, we have $\Ext^1_{dw}(U,
\Sigma^{-1}C)=\Ext^1_{dw}(\Sigma U, C)$ and the last group equals 0
since $\Sigma U\in\widetilde{\mathcal{U}}_\mathcal{X}$. Thus
$\Ext^1_{dw}(U, \Sigma^{-1}C)=0$, and so $C\in\widehat{\mathcal{W}}$
by Lemma \ref{lem2.1}.

Last we assume that $\Ext^1_{\text{Ch}(R)}(C, W)=0$ for all
$W\in\widehat{\mathcal{W}}$. We will show
$C\in\widetilde{\mathcal{U}}_\mathcal{X}$. Since for any $R$-module
$V\in\mathcal{V}$, the disk $D^{n+1}(V)\in\widehat{\mathcal{W}}$, we
have $\Ext^1_R(C_n, V)\cong\Ext^1_{\text{Ch}(R)}(C, D^{n+1}(V))=0$,
and so $C_n\in\mathcal{U}$. Also since
$S^n(Y)\in\widehat{\mathcal{W}}$ for any $R$-module
$Y\in\mathcal{Y}$, we have $\Ext^1_{\text{Ch}(R)}(C, S^n(Y))=0$, and
so $\Ext^1_R(C_n/B_nC, Y)=0$ by \cite[Lemma 4.2]{G08}, which implies
that each $C_n/B_nC$ belongs to $\mathcal{X}$. But by using Lemma
\ref{lem3.1} we get that each $Z_nC\cong
C_{n+1}/B_{n+1}C\in\mathcal{X}$. Thus
$C\in\widetilde{\mathcal{U}}_\mathcal{X}$, as desired.
\end{proof}

Given two cotorsion pairs  $(\mathcal{U}, \mathcal{V})$ and
$(\mathcal{X}, \mathcal{Y})$   in $R$-Mod. Clearly,
$\mathcal{U}\subseteq\mathcal{X}$ if and only if
$\mathcal{Y}\subseteq\mathcal{V}$, with this in mind, we also have
the following result.

\begin{thm}\label{thm3.4} Let $(\mathcal{U}, \mathcal{V})$ and $(\mathcal{X},
\mathcal{Y})$ be two cotorsion pairs  with
$\mathcal{U}\subseteq\mathcal{X}$ in $R$-Mod. Then
$(^\bot(\widetilde{\mathcal{Y}}_\mathcal{V}),
\widetilde{\mathcal{Y}}_\mathcal{V})$ is a cotorsion pair in Ch($R$)
and $^\bot(\widetilde{\mathcal{Y}}_\mathcal{V})$ is the class of all
chain complexes $X$ for which each $X_n\in\mathcal{X}$ and for which
each map $X\rightarrow Y$ is null homotopic whenever
$Y\in\widetilde{\mathcal{Y}}_\mathcal{V}$.
\end{thm}
\begin{proof} It is dual to the proof of
Theorem \ref{thm3.3}.
\end{proof}

\begin{df} Given a class  of $R$-modules $\mathcal{A}$. We define
the following classes of chain complexes in Ch($R$).
\begin{enumerate}
  \item $dw\widetilde{\mathcal{A}}$ is the class of all chain complexes  $A$
with each degree $A_n\in \mathcal{A}$.
  \item $ex\widetilde{\mathcal{A}}$ is the class of all exact  chain complexes  $A$
with each degree $A_n\in \mathcal{A}$.
  \item $\widetilde{\mathcal{A}}$ is the class of all exact  chain complexes  $A$ with
each cycle $Z_nA\in \mathcal{A}$.
\end{enumerate}
\end{df}

The $\lq\lq dw$" is meant to stand for $\lq\lq$degreewise" while the
$\lq\lq ex$" is meant to stand for $\lq\lq$exact".

Moreover, if we are given any cotorsion pair $(\mathcal{U},
\mathcal{V})$ in $R$-Mod, then following \cite{G04} we will denote
$\widetilde{\mathcal{U}}^\bot$ by $dg\widetilde{\mathcal{V}}$ and
$^\bot\widetilde{\mathcal{V}}$ by $dg\widetilde{\mathcal{U}}$.

 The next two corollaries
are contained in \cite[Proposition 3.6]{G04}, and \cite[Proposition
3.3]{G08}, respectively, but the author considered them on a general
abelian category. Here we present short proofs of them for our case.

\begin{cor}\label{c04} Let $(\mathcal{U}, \mathcal{V})$ be a cotorsion pair in $R$-Mod. Then
$(\widetilde{\mathcal{U}}, dg\widetilde{\mathcal{V}})$ and
$(dg\widetilde{\mathcal{U}}, \widetilde{\mathcal{V}})$ are cotorsion
pairs in Ch($R$).
\end{cor}
\begin{proof} We just prove one of the statements since the other is dual. Note that
$(\mathcal{X}, \mathcal{Y})=(\mathcal{U}, \mathcal{V})$ is another
cotorsion pair with $\mathcal{U}\subseteq\mathcal{X}$. So, by
Theorem \ref{thm3.3}, $(\widetilde{\mathcal{U}},
\widetilde{\mathcal{U}}^\bot)$ is a cotorsion pair since clearly
$\widetilde{\mathcal{U}}=\widetilde{\mathcal{U}}_\mathcal{U}$.
\end{proof}

In the following, we take $(\mathcal{P}, \mathcal{M})$ and
$(\mathcal{M}, \mathcal{I})$ as the usual projective and injective
cotorsion pairs in $R$-Mod, where $\mathcal{P}$ denotes the class of
all projective $R$-modules, $\mathcal{M}$ denotes the class of all
$R$-modules, and $\mathcal{I}$ denotes the class of all injective
$R$-modules. Note that for any cotorsion pair $(\mathcal{U},
\mathcal{V})$ in $R$-Mod we always have inclusions
$\mathcal{P}\subseteq \mathcal{U}\subseteq \mathcal{M}$ and
$\mathcal{I}\subseteq \mathcal{V}\subseteq \mathcal{M}$.

\begin{cor}\label{c05} Let $(\mathcal{U}, \mathcal{V})$ be a cotorsion pair in $R$-Mod. Then
$(ex\widetilde{\mathcal{U}}, (ex\widetilde{\mathcal{U}})^\bot)$ and
$(^\bot(ex\widetilde{\mathcal{V}}), ex\widetilde{\mathcal{V}})$ are
cotorsion pairs in Ch($R$).
\end{cor}
\begin{proof} Again we will just prove one of the statements since the other is dual. Note that
the injective cotorsion pair $(\mathcal{M}, \mathcal{I})$ in $R$-Mod
is another one such that $\mathcal{U}\subseteq\mathcal{M}$. So, by
Theorem \ref{thm3.3}, $(ex\widetilde{\mathcal{U}},
(ex\widetilde{\mathcal{U}})^\bot)$ is a cotorsion pair since clearly
$ex\widetilde{\mathcal{U}}=\widetilde{\mathcal{U}}_\mathcal{M}$.
\end{proof}

\begin{rem}\label{re3.7} According to  \cite{YL}, the induced cotorsion pairs
$(\widetilde{\mathcal{U}}, dg\widetilde{\mathcal{V}})$ and
$(dg\widetilde{\mathcal{U}}, \widetilde{\mathcal{V}})$ are both
complete when the given cotorsion pair  $(\mathcal{U}, \mathcal{V})$
is hereditary and complete. In particular,
$(dg\widetilde{\mathcal{P}}, \widetilde{\mathcal{M}})$ and
$(\widetilde{\mathcal{M}}, dg\widetilde{\mathcal{I}})$ are complete
cotorsion pairs in Ch($R$), where $\widetilde{\mathcal{M}}$ denotes
the class of all exact chain complexes.
\end{rem}

\begin{rem} \label{re3.8} According to  \cite{G08}, if we have a cotorsion pair $(\mathcal{U},
\mathcal{V})$ in $R$-Mod, then $(dw\widetilde{\mathcal{U}},
{(dw\widetilde{\mathcal{U}})}^\bot)$ and
$(^\bot{(dw\widetilde{\mathcal{V}})}, dw\widetilde{\mathcal{V}})$
are cotorsion pairs in Ch($R$).
\end{rem}

  The following result shows that there are
intimate connections of completeness between the induced cotorsion
pairs $(dw\widetilde{\mathcal{U}},
{(dw\widetilde{\mathcal{U}})}^\bot)$ and
$(ex\widetilde{\mathcal{U}}, {(ex\widetilde{\mathcal{U}})}^\bot)$.

\begin{thm} \label{rhm3.9}
Assume that $(\mathcal{U}, \mathcal{V})$ is a hereditary cotorsion pair
in $R$-Mod. Then $(dw\widetilde{\mathcal{U}},
(dw\widetilde{\mathcal{U}})^\bot)$ is complete if and only if
$(ex\widetilde{\mathcal{U}}, (ex\widetilde{\mathcal{U}})^\bot)$ is
complete.
\end{thm}
\begin{proof}  ($\Rightarrow$). Since the cotorsion pair $(dw\widetilde{\mathcal{U}},
(dw\widetilde{\mathcal{U}})^\bot)$ is complete, for any chain
complex $C$, there exists an exact sequence $0\rightarrow
V\rightarrow U\rightarrow C\rightarrow 0$ such that $U\in
dw\widetilde{\mathcal{U}}$ and $V\in
(dw\widetilde{\mathcal{U}})^\bot$. By  Remark \ref{re3.7}, there is
an exact sequence $0\rightarrow I\rightarrow E\rightarrow
U\rightarrow 0$ with  $E$ exact and $I\in
dg\widetilde{\mathcal{I}}$. Now we consider the pull-back diagram as
follows:
\[
\begin{CD}
@.0 @.  0 \\
@.  @VVV @VVV \\
@.I @= I \\
@.  @VVV @VVV \\
0 @>>> X @>>> E @>>> C@>>> 0\\
@.  @VVV @VVV @|\\
0@>>> V @>>> U @>>> C@>>> 0 \\
@.  @VVV @VVV \\
@.0 @.  0
\end{CD}
\]
Since $I\in dg\widetilde{\mathcal{I}}$, we have $I\in
(ex\widetilde{\mathcal{U}})^\bot$. Clearly, we have $V\in
(ex\widetilde{\mathcal{U}})^\bot$ since $V\in
(dw\widetilde{\mathcal{U}})^\bot$, and so the exactness of the
leftmost column of the above diagram implies $X\in
(ex\widetilde{\mathcal{U}})^\bot$. By hypothesis again, there is an
exact sequence $0\rightarrow V'\rightarrow U'\rightarrow
E\rightarrow 0$ such that $U'\in dw\widetilde{\mathcal{U}}$ and
$V'\in (dw\widetilde{\mathcal{U}})^\bot$. Again consider the
following pull-back diagram:
\[
\begin{CD}
@.0 @.  0 \\
@.  @VVV @VVV \\
@.V' @= V' \\
@.  @VVV @VVV \\
0 @>>> Y @>>> U' @>>> C@>>> 0\\
@.  @VVV @VVV @|\\
0@>>> X @>>> E @>>> C@>>> 0 \\
@.  @VVV @VVV \\
@.0 @.  0
\end{CD}
\]
Since $E$ is exact and $V'\in (dw\widetilde{\mathcal{U}})^\bot$ is
easily seen exact, we get that $U'$ is exact and so $U'\in
ex\widetilde{\mathcal{U}}$. Furthermore, since $V'\in
(dw\widetilde{\mathcal{U}})^\bot\subseteq
(ex\widetilde{\mathcal{U}})^\bot$ and $X\in
(ex\widetilde{\mathcal{U}})^\bot$, the exactness of the leftmost
column of the above diagram implies $Y\in
(ex\widetilde{\mathcal{U}})^\bot$. Now the second exact row of the
above diagram implies that the cotorsion pair
$(ex\widetilde{\mathcal{U}}, (ex\widetilde{\mathcal{U}})^\bot)$ has
enough projectives, and so it is complete.

($\Leftarrow$). Let $C$ be any chain complex.  Then there is an
exact sequence $0\rightarrow H\rightarrow G\rightarrow C\rightarrow
0$ with $G\in ex\widetilde{\mathcal{U}}$ and $H\in
(ex\widetilde{\mathcal{U}})^\bot$ since the cotorsion pair
$(ex\widetilde{\mathcal{U}}, (ex\widetilde{\mathcal{U}})^\bot)$ is
complete. By Remark \ref{re3.7}, there is  an exact sequence
$0\rightarrow H\rightarrow E\rightarrow P\rightarrow 0$ with $E$
exact and $P\in dg\mathcal{\widetilde{P}}$. Consider the following
push-out diagram:
\[
\begin{CD}
@.0 @.  0\\
@. @VVV  @VVV\\
0 @>>> H @>>> G @>>> C@>>> 0\\
@.  @VVV @VVV @|\\
0@>>> E @>>> D @>>> C@>>> 0 \\
@.  @VVV @VVV \\
@.P @=  P\\
@.  @VVV @VVV \\
@.0 @.  0
\end{CD}
\]
Since $G\in ex\widetilde{\mathcal{U}}\subseteq
dw\widetilde{\mathcal{U}}$  and $P$ is easily seen in
$dw\widetilde{\mathcal{U}}$, we get that $D\in
dw\widetilde{\mathcal{U}}$. Since the cotorsion pair
$(ex\widetilde{\mathcal{U}}, (ex\widetilde{\mathcal{U}})^\bot)$ is
complete, there is an exact sequence $0\rightarrow E\rightarrow
H'\rightarrow G'\rightarrow 0$ with $H'\in
(ex\widetilde{\mathcal{U}})^\bot$ and $G'\in
ex\widetilde{\mathcal{U}}$. Now consider the following push-out
diagram:
\[
\begin{CD}
@.0 @.  0\\
@. @VVV  @VVV\\
0 @>>> E @>>> D @>>> C@>>> 0\\
@.  @VVV @VVV @|\\
0@>>> H' @>>> B @>>> C@>>> 0 \\
@.  @VVV @VVV \\
@.G' @=  G'\\
@.  @VVV @VVV \\
@.0 @.  0
\end{CD}
\]
Since $E$ and $G'$ are exact, so is $H'$. Thus it follows from
\cite[Lemma 7.4.1]{EJ} that $H'\in
(ex\widetilde{\mathcal{U}})^\bot\cap\widetilde{\mathcal{M}}=(dw\widetilde{\mathcal{U}})^\bot$.
  It is not hard to see that $B\in dw\widetilde{\mathcal{U}}$.
This proves that the cotorsion pair $(dw\widetilde{\mathcal{U}},
(dw\widetilde{\mathcal{U}})^\bot)$  is complete.
\end{proof}

Dually, we have the following result without giving its proof.

\begin{thm} \label{rhm3.11}
Assume that $(\mathcal{U}, \mathcal{V})$ is a hereditary cotorsion pair
in $R$-Mod. Considering the  statements below. Then
$({^\bot(dw\widetilde{\mathcal{V}})}, dw\widetilde{\mathcal{V}})$ is
complete if and only if $({^\bot(ex\widetilde{\mathcal{V}})},
ex\widetilde{\mathcal{V}})$ is complete.
\end{thm}

Assume that the given cotorsion pairs $(\mathcal{U}, \mathcal{V})$
and $(\mathcal{X}, \mathcal{Y})$  with
$\mathcal{U}\subseteq\mathcal{X}$ in $R$-Mod are hereditary, then it
is easily seen that the induced cotorsion pairs
$(\widetilde{\mathcal{U}}_\mathcal{X},
(\widetilde{\mathcal{U}}_\mathcal{X})^\bot)$ and
$(^\bot(\widetilde{\mathcal{Y}}_\mathcal{V}),
\widetilde{\mathcal{Y}}_\mathcal{V})$ are both hereditary. In the
following, we are ready to show that our induced cotorsion pairs are
complete under certain conditions. We will use a generalized
version, of a well-known result of Eklof and Trlifaj \cite[Theorem
10]{ET}, which says that every cotorsion pair $(\mathcal{A},
\mathcal{B})$ in any Grothendieck category with enough projectives
is complete if it is cogenerated by a set, see \cite[Section
6]{Hov}. We say a cotorsion pair $(\mathcal{A}, \mathcal{B})$ in an
abelian category $\mathcal{C}$ is cogenerated by a set if there is a
set $\mathcal{S}\subseteq\mathcal{A}$ such that
$\mathcal{S}^\bot=\mathcal{B}$.

\begin{prop}\label{Prop3.12} Let $(\mathcal{U}, \mathcal{V})$ and $(\mathcal{X},
\mathcal{Y})$ be two cotorsion pairs  with
$\mathcal{U}\subseteq\mathcal{X}$ in $R$-Mod. If $(\mathcal{U},
\mathcal{V})$ is cogenerated by a set $\{A_i|i\in I\}$, and
$(\mathcal{X}, \mathcal{Y})$ is cogenerated by a set $\{B_j|i\in
J\}$, then the induced cotorsion pair
$(^\bot(\widetilde{\mathcal{Y}}_\mathcal{V}),
\widetilde{\mathcal{Y}}_\mathcal{V})$ is cogenerated by the set $
\mathcal{S}=\{S^n(R)|n\in \mathbb{Z}\}\cup\{S^n(A_i)|n\in
\mathbb{Z}, i\in I\}\cup\{D^n(B_j)|n\in \mathbb{Z}, j\in J\}$, and
so it is complete.
\end{prop}
\begin{proof}
Dual to the proof of Theorem \ref{thm3.3}, we can prove that each
sphere $S^n(U)\in{^\bot(\widetilde{\mathcal{Y}}_\mathcal{V})}$
whenever $U\in\mathcal{U}$, and each disk
$D^n(X)\in{^\bot(\widetilde{\mathcal{Y}}_\mathcal{V})}$ whenever
$X\in\mathcal{X}$. Thus we have
$\mathcal{S}\subseteq{^\bot(\widetilde{\mathcal{Y}}_\mathcal{V})}$,
and so
$\mathcal{S}^\bot\supseteq(^\bot(\widetilde{\mathcal{Y}}_\mathcal{V}))^\bot=\widetilde{\mathcal{Y}}_\mathcal{V}$.
To see the reverse inclusion, now suppose $Y\in\mathcal{S}^\bot$.
Then $\Ext^1_{\text{Ch}(R)}(S^n(A_i), Y)=0$ for all $i\in I$. Since
$\Ext^1_{\text{Ch}(R)}(S^n(A_i), Y)\cong\Ext^1_R(A_i, Z_nY)$ by
\cite[Lemma 3.1]{G04}, and the cotorsion pair $(\mathcal{U},
\mathcal{V})$ is cogenerated by $\{A_i|i\in I\}$, we get that each
$Z_nY\in\mathcal{V}$.

Next we show that $Y$ is exact. If we apply $\Hom_{\text{Ch}(R)}(-,
Y)$ to the short exact sequence $0\rightarrow S^{n-1}(R)\rightarrow
D^n(R)\rightarrow S^n(R)\rightarrow 0$, then we have an induced
exact sequence of abelian groups
$$\Hom_{\text{Ch}(R)}(D^n(R),
Y)\rightarrow \Hom_{\text{Ch}(R)}(S^{n-1}(R),
Y)\rightarrow\Ext^1_{\text{Ch}(R)}(S^n(R), Y)=0.$$ This means that
every chain map $S^{n-1}(R)\rightarrow Y$ can be extended to
$D^n(R)$. So $Y$ is exact by \cite[Lemma 2.4]{G08}.

It is left to show that each degree $Y_n$ of $Y$ belongs to
$\mathcal{Y}$ for any integer $n\in \mathbb{Z}$. By \cite[Lemma
3.1]{G04}, we have $\Ext^1_R(B_j,
Y_n)\cong\Ext^1_{\text{Ch}(R)}(D^n(B_j), Y)=0$ for all $j\in J$.
Thus $Y_n\in\mathcal{Y}$ since the cotorsion pair $(\mathcal{X},
\mathcal{Y})$ is cogenerated by $\{B_j|i\in J\}$. This shows that $
\mathcal{S}$ cogenerates the cotorsion pair
$(^\bot(\widetilde{\mathcal{Y}}_\mathcal{V}),
\widetilde{\mathcal{Y}}_\mathcal{V})$.
\end{proof}

\begin{prop}\label{Prop3.13} Let $(\mathcal{U}, \mathcal{V})$ and $(\mathcal{X},
\mathcal{Y})$ be two cotorsion pairs  with
$\mathcal{U}\subseteq\mathcal{X}$ in $R$-Mod. If both $(\mathcal{U},
\mathcal{V})$ and $(\mathcal{X}, \mathcal{Y})$ are cogenerated by
sets, then so is the induced cotorsion pair
$(\widetilde{\mathcal{U}}_\mathcal{X},
(\widetilde{\mathcal{U}}_\mathcal{X})^\bot)$, and so it is complete.
\end{prop}
\begin{proof} Since $(\mathcal{U},
\mathcal{V})$ and $(\mathcal{X}, \mathcal{Y})$ are cogenerated by
sets, the two classes $\mathcal{U}$ and $\mathcal{X}$ are
deconstructible by \cite{St}. Then $dw\widetilde{\mathcal{U}}$ and
$\widetilde{\mathcal{X}}$ are deconstructible classes by
\cite[Theorem 4.2]{St}, and so
$\widetilde{\mathcal{U}}_\mathcal{X}=dw\widetilde{\mathcal{U}}\cap\widetilde{\mathcal{X}}$
is deconstructible by \cite[Proposition 2.9(2)]{St}. This implies
that the cotorsion pair $(\widetilde{\mathcal{U}}_\mathcal{X},
(\widetilde{\mathcal{U}}_\mathcal{X})^\bot)$ is cogenerated by a
set, and so it is complete.
\end{proof}

It is know that every Kaplansky class which is closed under well
ordered direct limits is deconstructible, and every deconstructible
class is Kaplansky. However, both of the converse do not hold in
general (see \cite[Lemmas 6.7 and 6.9, and Examlpe 6.8]{HT}). In the
following we present two examples relating to Kaplansky classes  as
applications of Proposition \ref{Prop3.13}.

\begin{exa} Let $(\mathcal{U}, \mathcal{V})$ and $(\mathcal{X},
\mathcal{Y})$ be two cotorsion pairs  with
$\mathcal{U}\subseteq\mathcal{X}$ in $R$-Mod. If $\mathcal{U}$ and
$\mathcal{X}$ are Kaplansky classes which are both closed under well
ordered direct limits, and $\mathcal{X}$ is resolving, then
$\widetilde{\mathcal{U}}_\mathcal{X}$ is a Kaplansky class of chain
complexes which is closed under  well ordered direct limits. Thus,
the cotorsion pair $(\widetilde{\mathcal{U}}_\mathcal{X},
(\widetilde{\mathcal{U}}_\mathcal{X})^\bot)$ is cogenerated by a
set, and so it is complete.
\end{exa}
\begin{proof} Let $\kappa_1, \kappa_2$ be the cardinal numbers such that
$\mathcal{U}$ is $\kappa_1$-Kaplansky, and $\mathcal{X}$ is
$\kappa_2$-Kaplansky. Let $\kappa$ be a cardinal number larger than
$\text{max}\{\kappa_1, \kappa_2, \omega, \text{Card}(R)\}$. In the
following, we wish to show that
$\widetilde{\mathcal{U}}_\mathcal{X}$ is $\kappa$-Kaplansky. So
assume that $U\in\widetilde{\mathcal{U}}_\mathcal{X}$ and $S$ is a
subset of $U$ with $\text{Card}(S)\leq\kappa$. We show that there
exists a chain subcomplex $W$ of $U$ such that $S\subseteq W$,
$\text{Card}(W)\leq\kappa$, and $W$ and $U/W$ are contained in
$\widetilde{\mathcal{U}}_\mathcal{X}$.

Note that $\widetilde{\mathcal{U}}_\mathcal{M}$ and
$\widetilde{\mathcal{X}}_\mathcal{X}$ are two $\kappa$-Kaplansky
classes  of chain complexes by \cite[Theorem 3.4]{AH}. Thus there
exists a chain subcomplex $U^1$ of $U$ such that $S\subseteq U^1$,
$\text{Card}(U^1)\leq\kappa$, and $U^1$ and $U/U^1$ are contained in
$\widetilde{\mathcal{U}}_\mathcal{M}$. Again, since
$\widetilde{\mathcal{X}}_\mathcal{X}$ is $\kappa$-Kaplansky, there
exists a chain subcomplex $U^2$ of $U$ such that $U^1\subseteq U^2$,
$\text{Card}(U^2)\leq\kappa$, and $U^2$ and $U/U^2$ are contained in
$\widetilde{\mathcal{X}}_\mathcal{X}$. Thus we will construct
inductively $\{U^i\}_{i\in \mathbb{N}}$ of chain subcomplexes of
$U$, satisfying the following three properties:

\begin{itemize}
  \item For any two  integers $i, j\in \mathbb{N}$ with
  $i<j$, $U^i$ is a chain subcomplex of $U^j$;
  \item $U^i$ satisfies
$\text{Card}(U^i)\leq\kappa$, and $U^i,
U/U^i\in\widetilde{\mathcal{U}}_\mathcal{M}$ whenever $i\in
\mathbb{N}$ is odd;
  \item  $U^i$ satisfies
$\text{Card}(U^i)\leq\kappa$, and $U^i,
U/U^i\in\widetilde{\mathcal{X}}_\mathcal{X}$ whenever $i\in
\mathbb{N}$ is even.
\end{itemize}

If we take $W=\varinjlim_{i\in \mathbb{N}}U^i=\bigcup_{i\in
\mathbb{N}}U^i$, then we see that the complex $W$ is exact because
of exactness of each $U^i$. Clearly, $W_n=\varinjlim_{i\in
\mathbb{N}}U_n^i=\bigcup_{i\in \mathbb{N}}U_n^i=\bigcup_{i\in
\mathbb{N}}U_n^{2i-1}$. But by constructions, we have each
$U^{2i-1}\in\widetilde{\mathcal{U}}_\mathcal{M}$. In particular,
$W_n\in\mathcal{U}$. Furthermore, each $Z_nW\in\mathcal{X}$ since
$Z_nW=\varinjlim_{i\in \mathbb{N}}Z_nU^i=\bigcup_{i\in
\mathbb{N}}Z_nU^i=\bigcup_{i\in \mathbb{N}}Z_nU^{2i}$ and
$U^{2i}\in\widetilde{\mathcal{X}}_\mathcal{X}$ by constructions.
Therefore the chain subcomplex $W$ of $U$ satisfies
$W\in\widetilde{\mathcal{U}}_\mathcal{X}$, $S\subseteq W$, and of
course  $\text{Card}(W)\leq\kappa$. To finish the proof, we need
only argue that $U/W\in\widetilde{\mathcal{U}}_\mathcal{X}$. It
follows from the short exact sequence $0\rightarrow W\rightarrow U
\rightarrow U/W\rightarrow 0$ of chain complexes that $U/W$ is
exact. Also one can check easily that $U/W=U/(\varinjlim_{i\in
\mathbb{N}}U^i)\cong\varinjlim_{i\in \mathbb{N}}U/U^i$, and then an
easy computation shows that $U_n/W_n\in\mathcal{U}$ and
$Z_n(U/W)\in\mathcal{X}$. This shows that
$\widetilde{\mathcal{U}}_\mathcal{X}$ is a Kaplansky class. Note
that $\widetilde{\mathcal{U}}_\mathcal{X}$ is also closed under well
ordered direct limits, and so it is deconstructible. This implies
that the cotorsion pair $(\widetilde{\mathcal{U}}_\mathcal{X},
(\widetilde{\mathcal{U}}_\mathcal{X})^\bot)$ is cogenerated by a
set.
\end{proof}

Recall that an exact sequence $0\rightarrow L\rightarrow M$ of
$R$-modules is pure if $0\rightarrow N\otimes_RL\rightarrow
N\otimes_RM$ is exact for any right $R$-module $N$. We say a
submodule $L$ of $M$ is pure if the sequence $0\rightarrow
L\rightarrow M$ is pure exact. Similarly, we have the notion of a
pure chain subcomplexes, but this will use tensor product of  chain
complexes (see \cite{GR}). It is easy to see that if a class of
$R$-modules is closed under pure submodules and cokernels of pure
monomorphisms, then it is Kaplansky, also it is deconstructible.

\begin{exa} Let $(\mathcal{U}, \mathcal{V})$ and $(\mathcal{X},
\mathcal{Y})$ be two cotorsion pairs  with
$\mathcal{U}\subseteq\mathcal{X}$ in $R$-Mod. Assume that
$\mathcal{U}$ and $\mathcal{X}$ are  both closed under pure
submodules and cokernels of pure monomorphisms. Then
$\widetilde{\mathcal{U}}_\mathcal{X}$ is closed under pure
subcomplexes and cokernels of pure monomorphisms, and so the
cotorsion pair $(\widetilde{\mathcal{U}}_\mathcal{X},
(\widetilde{\mathcal{U}}_\mathcal{X})^\bot)$ is cogenerated by a
set. Thus $(\widetilde{\mathcal{U}}_\mathcal{X},
(\widetilde{\mathcal{U}}_\mathcal{X})^\bot)$ is complete..
\end{exa}
\begin{proof} Suppose that the exact seqence $0\rightarrow V\rightarrow U\rightarrow
U/V\rightarrow 0$ is pure in $\text{Ch}(R)$ with
$U\in\widetilde{\mathcal{U}}_\mathcal{X}$. Then chain complexes $V$
and $U/V$ are exact as shown in proof of \cite[Lemma 2.7]{WL}, and
so for each $n\in\mathbb{Z}$,   $0\rightarrow V_n\rightarrow
U_n\rightarrow U_n/V_n\rightarrow 0$ and $0\rightarrow
Z_nV\rightarrow Z_nU\rightarrow Z_n(U/V)\rightarrow 0$ are pure
exact in $R$-Mod by \cite[Lemmas 2.6 and 3.7]{WL}. It is easily seen
that $V$ and $U/V$ are in $\widetilde{\mathcal{U}}_\mathcal{X}$.
This shows that $\widetilde{\mathcal{U}}_\mathcal{X}$ is closed
under pure subcomplexes and cokernels of pure monomorphisms. Now it
follows that the cotorsion pair
$(\widetilde{\mathcal{U}}_\mathcal{X},
(\widetilde{\mathcal{U}}_\mathcal{X})^\bot)$ is cogenerated by a set
since $\widetilde{\mathcal{U}}_\mathcal{X}$ is clearly 
deconstructible.
\end{proof}
%%%%%%%%%%%%%%%%%%%%%%%%%%%%%%%%%%%%%%%%%%%%%%%%%%%%%%%%%%%%%%%%%%%%%%%%


\begin{thebibliography}{40}
\bibitem[AERO]{AERO} S. T. Aldrich, E. E. Enochs, J. R. Garc\'{\i}a Rozas and L. Oyonarte, Covers and envelopes in Grothendieck categories: flat covers
of complexes with applications.   J. Algebra \textbf{243} (2001),
615-630.


\bibitem[AH]{AH} J. Asadollahi and R. Hafezi, Kaplansky classes
and cotorsion theories of complexes. Commun.
Algebra  \textbf{42} (2014), 1953-1964. 

\bibitem[BBE]{BBE} L. Bican, R. El Bashir,  E. E. Enochs, All modules
have flat covers. Bull. Lond. Math. Soc. \textbf{33} (2001),
385-390.


\bibitem[BEIJR]{BEIJR} D. Bravo, E. E. Enochs, A. C. Iacob, O. M. G.
Jenda, and J. Rada, Cotorsion pairs in C($R$-Mod). Rocky Mount. J. Math. (6) \textbf{42} (2012), 1787-1802.

\bibitem[EER]{EER} E. E. Enochs, S. Estrada, and J. R. Garc\'{\i}a Rozas,
Locally projective monoidal model structure for complexes of
quasi-coherent sheaves an $P^1(k)$. J. London Math. Soc. (2)
\textbf{77} (2008), 253-269.


\bibitem[EJ]{EJ} E. E. Enochs and O. M. G. Jenda, \emph{Relative Homological Algebra, Volume 2},  de
Gruyter Expositions in Mathematics, 54. Walter de Gruyter GmbH \&
Co. KG, Berlin, 2011.

\bibitem[EL]{EL} E. E. Enochs, J. A. L$\acute{o}$pez-Ramos, Kaplansky classes. Rend
Sem. Mat. Univ. Padova \textbf{107} (2002), 67-79.


\bibitem[ET]{ET} P. C. Eklof and J. Trlifaj, How to make Ext vanish.
Bull. London Math. Soc. \textbf{33} (2001), 41-51.

\bibitem[GR]{GR} J. R. Garc\'{\i}a Rozas, \emph{Covers and envelope in the
category of complexes of modules}. Research Notes in Mathematics no.
407, Chapman \& Hall/CRC, Boca Raton FL, 1999.

\bibitem[G04]{G04} J. Gillespie, The flat model structure on Ch($R$). Trans.
Amer. Math. Soc.  \textbf{356} (2004), 3369-3390.


\bibitem[G08]{G08} J. Gillespie, Cotorsion pairs and degreewise
homological model structures. Homology, Homotopy and Applications
(1) \textbf{10} (2008), 283-304.

\bibitem[GT]{GT} R. G\"{o}bel, and J. Trlifaj, \emph{Approximations
and endmorphism algebras of modules}.  de Gruyter Exposition in
Mathematics, 41. Walter De Gruyter GmbH \& Co. KG, Berlin, 2006.

\bibitem[HT]{HT} D. Herbera, and J. Trlifaj, Almost free modules and
Mittag-Leffler conditions, Advances in Mathematics (6) \textbf{229}
(2012), 3436-3467.


\bibitem[Hol]{Hol} H. Holm, Gorenstein homological dimensions. J. Pure Appl. Algebra
\textbf{189} (2004),  167-193.



\bibitem[Hov]{Hov} M. Hovey, Cotorsion pairs, model
category structures, and representation theory. Math. Zeit.
\textbf{241} (2002), 553-592.


\bibitem[K]{K} I. Kaplansky, Projective modules. Ann. Math.
\textbf{68} (1958), 372-377.


\bibitem[Sa]{Sa} L. Salce, \emph{Cotorsion theories for
abelian groups}. Symposia Math. \textbf{23}, 11-32. Acdemic Press,
New York, 1979.

\bibitem[St]{St} J. Stovicek, Deconstructibility and the Hill lemma in Grothendieck
categories, Forum Math. \textbf{25} (2013), 193-219.

\bibitem[WL]{WL} Z. P. Wang, and Z. K. Liu, Complete cotorsion pairs in the category of
complexes. Turk. J. Math. \textbf{37} (2013), 852-862.
\bibitem[YL]{YL} G. Yang, and Z. K. Liu, Cotorsion pairs and model structures on
Ch($R$). Proceedings of the Edinburgh Mathematical Society
\textbf{54} (2011), 783-797.




\end{thebibliography}
\end{document}